\documentclass{amsart}
\usepackage {amsmath, amscd}
\usepackage {amssymb}
\usepackage{enumitem}

\numberwithin{equation}{section}

\def\<{\langle}
\def\>{\rangle}

\def\K{{\bf K}}

\def\DD{{\mathcal D}}
\def\EE{{\mathcal E}}

\def\HH{{\mathcal H}}
\def\II{{\mathcal I}}
\def\JJ{{\mathcal J}}
\def\KK{{\mathcal K}}
\def\LL{{\mathcal L}}

\def\XX{{\mathcal X}}
\def\YY{{\mathcal Y}}

\def\bbC{\mathbb{C}}

\def\bbN{\mathbb{N}}

\def\bbD{\mathbb{D}}
\def\bbT{\mathbb{T}}

\newcommand{\Toe}{\mathbf{T}}

\renewcommand{\<}{\langle}
\renewcommand{\>}{\rangle}

\newtheorem{lemma}{Lemma}[section]
\newtheorem{proposition}[lemma]{Proposition}%[section]
\newtheorem{theorem}[lemma]{Theorem}%[section]

\theoremstyle{definition}

\newtheorem{example}[lemma]{Example}

\title[Nonextreme de Branges--Rovnyak spaces as models]{Nonextreme de Branges--Rovnyak spaces as models for contractions}

\author{Javad Mashreghi}

\address{D\'epartement de Math\'ematiques et de Statistique, Universit\'e Laval, Qu\'ebec, QC, Canada G1K 7P4}

\email{javad.mashreghi@mat.ulaval.ca}

\author{Dan Timotin}

\address{Institute of Mathematics of the Romanian Academy, P.O. Box 1-764, Bucharest 
014700, Romania}
\email{Dan.Timotin@imar.ro}

\thanks{The first author is supported by a grant from NSERC. The second author is partially supported by a grant of the Romanian National Authority for Scientific
Research, CNCS Ð UEFISCDI, project number PN-II-ID-PCE-2011-3-0119.}

\begin{document}

\begin{abstract}
The de Branges--Rovnyak spaces are known to provide an alternate functional model for contractions on a Hilbert space, equivalent to the Sz.-Nagy--Foias model. The scalar de Branges--Rovnyak spaces $\HH(b)$ have essentially different properties, according to whether the defining function $b$ is or not extreme in the unit ball of $H^\infty$. For $b$ extreme the model space is just $\HH(b)$, while for $b$ nonextreme  an additional construction is required. In the present paper we identify the  precise class of contractions which have as a model $\HH(b)$ with $b$ nonextreme.
\end{abstract}

\maketitle

\section{Introduction}

In order to understand better operators on a Hilbert space, one often tries to find models for certain classes; that is, a subclass of concrete operators with the property that any given operator from the class is unitarily equivalent to an element of the subclass. The typical example is given by  normal operators, which by the spectral theorem have multiplication operators on Lebesgue spaces as models.

Going beyond normal operators, there is an extensive theory dealing with models
for contractions. The most elaborate form is the Sz.-Nagy--Foias
theory~\cite{SF}, that we will shortly describe in the next section. About the
same time another model had been  devised by de Branges and then developed in
detail in~\cite{BR1,BR}; its main feature was the extensive use of contractively
included subspaces. It turned out in the end that the models are equivalent; an
explanation of the relation can be found in~\cite{BK, NV}. One should also note
that these so called de Branges--Rovnyak spaces have received new attention in
the last years, representing an active area of research (see, for
instance,~\cite{A, BBH, BFM, CGR, CR,   J}).
There is also an upcoming book on the subject~\cite{FMb}.

In the theory developed by de Branges and Rovnyak, an important starting point
is provided by the so called scalar case, when the spaces involved are nonclosed
subspaces of the Hardy space $H^2$. These spaces are determined by a function
$b$ in the unit ball of $H^\infty$, and the usual notation is $\HH(b)$; later
their theory has been extensively developed~\cite{BBH0, F, FM1, FM2, LS1, LS2,
LS3, Sa1}, an important role being played by the basic monograph of
Sarason~\cite{Sa}. It turns out that the study splits quite soon in two disjoint
cases, according to whether $b$ is or not an extreme point of the unit ball of
$H^\infty$.

From the point of view of model operators, the scalar case corresponds to the situation when the defect spaces of the contraction (see next section for precise definitions) have dimension~1. An important difference appears between the two situations: when $b$ is extreme, the model space is $\HH(b)$ itself, and the model operator the backward shift; but when $b$ is not extreme, the model space contains pairs of functions, only the first one being in $\HH(b)$, and the model operator acts in a more complicated way.

A natural question then appears: in the nonextreme space, can one also view $\HH(b)$ itself     as a model space (and the backward shift as a model operator) for a certain class of contractions? The present paper answers this question in the affirmative: we give in Theorem~\ref{th:main} precise necessary and sufficient conditions for a contraction on a Hilbert space to be unitarily equivalent to the backward shift acting on some space $\HH(b)$ with $b$ nonextreme. However, we should add that the description is rather involved, and different rather distinct functions $b$ may lead to unitarily equivalent models.

The plan of the paper is the following. After giving the necessary preliminaries in Section~2, we proceed to find necessary conditions for a contraction $T$ to be unitarily equivalent to the backward shift acting on some $\HH(b)$ with $b$ nonextreme. Two of these are rather immediate (see Section~3), and a third one is not hard to find (this is done in Section~4). The last decisive fourth condition requires more work, its discussion being the content of Sections 5 and 6. The main result is stated in Section~7, while Section~8 discusses to what extent is the function $b$ determined by the contraction.

\section{Preliminaries}

\subsection{General notations}

We will use the standard notations $L^2$ for the Lebesgue space of square integrable functions on the unit circle $\bbT$ and $H^2$ for the Hardy space, which may be alternately considered either as a closed subspace of $L^2$ or a space of analytic functions in the unit disc $\bbD$. We will meet also their vector valued variants $L^2(\EE)$ and $H^2(\EE)$, with $\EE$ a Hilbert space. Multiplication with $e^{it}$ on $L^2$ will be denoted by $Z$ and its restriction to $H^2$ by $S$; for their analogues in the vector valued spaces we will use bold letters $\mathbf{Z}$ and $\mathbf{S}$ respectively  (the space $\EE$ can be deduced from the context). The action of these operators on the Fourier coefficients of a function explains why $\mathbf{Z}$ is also called the \emph{bilateral shift}  and $\mathbf{S}$ the \emph{unilateral shift}.

The Hardy algebra $H^\infty$ of all bounded analytic functions in $\bbD$ acts by multiplication on $H^2$; the corresponding operator valued objects are analytic functions in $\bbD$ with values in $\LL(\EE_1, \EE_2)$ (the linear bounded operators); they map $H^2(\EE_1)$ into $H^2(\EE_2)$. In fact, we will only meet contractive analytic functions $\Theta:\bbD\to \LL(\EE_1, \EE_2)$, whose values are contractions from $\EE_1$ to $\EE_2$. Such a function can be decomposed as 
\[
\Theta(\lambda)= \begin{pmatrix}
\Theta_0(\lambda) & 0\\ 0 & W
\end{pmatrix}: \EE_1'\oplus \EE_1''\to \EE_2'\oplus \EE_2'',
\]
where $\EE_i=\EE_i'\oplus \EE_i''$ ($i=1,2$), $W$ is a unitary constant, and $\Theta_0$ is \emph{pure}, that is, it has no constant part; $\Theta_0$ is called the \emph{pure part} of $\Theta$. 

\subsection{The Sz.-Nagy--Foias model and related questions}

If $H$ is a Hilbert space, we  denote by $\LL(H)$ the algebra of all bounded operators acting on $H$.
Let then $T\in\LL(H)$ be a contraction, that is, $\|T\|\le 1$. We define $D_T=(I-T^*T)^{1/2}$ and $\DD_T=\overline{D_T H}$. Obviously $T$ is unitary if and only if $\DD_T=\DD_{T^*}=\{0\}$. For a general contraction, there exists a unique decomposition $H=H_u\oplus H_c$, where $H_u$ and $H_c$ are invariant with respect to $T$ (and hence reducing), $T|H_u$ is unitary, while $T|H_c$ is \emph{completely nonunitary (c.n.u.)}; that is, it has no reducing space on which it is unitary.

A \emph{dilation} $\widehat T$ of $T$ is an operator acting on a space $\widehat H\supset H$, such that $P_H\widehat T^n|H=T^n$ for all $n\ge 0$. Such a dilation is \emph{minimal} if $\bigvee_{n\ge 0}\widehat T^nH=\widehat H$. Any dilation $\widehat T$ of $T$ ``contains'' a minimal one: it suffices to restrict $\widehat T$ to its invariant subspace spanned by $H$.

The Sz.-Nagy dilation theorem states that any contraction has a minimal isometric dilation, which is unique up to a unitary equivalence that is the identity on $H$; a similar result is true for minimal unitary dilations.

The structure of unitary operators can be rather well described by means of the spectral theorem. On the other hand, for a c.n.u. contraction a structure description is given by the ``model'' theory of Sz.-Nagy and Foias~\cite{SF} that we describe below. A central role is played by the notion of characteristic function. The characteristic function of a completely nonunitary contraction $T\in\LL(H)$ is the contractive valued analytic function $\Theta(\lambda):\DD_T\to \DD_{T^*}$, defined by
\[
\Theta(\lambda)=-T+\lambda D_{T^*}(I-\lambda T^*)^{-1}D_T|\DD_T,\quad \lambda\in\bbD.
\]
The main result states that $T$ is unitarily equivalent with its model $\mathbf{S}_{\Theta}\in \LL(\mathbf{K}_\Theta)$, defined as follows:
\begin{equation*}
\begin{split}
\mathbf{K}_\Theta&= (H^2(\DD_{T^*})\oplus \overline{(I-\Theta^*\Theta)^{1/2}L^2(\DD_T)})\ominus \{\Theta h\oplus (I-\Theta^*\Theta)^{1/2}h: h\in H^2(\DD_T)\},\\
\mathbf{S}_\Theta& = P_{\mathbf{K}_\Theta} (\mathbf{S}\oplus \mathbf{Z})| \mathbf{K}_\Theta.
\end{split}
\end{equation*}
Also, $\mathbf{K}_\Theta$ is invariant with respect to $\mathbf{S}^*\oplus \mathbf{Z}^*$, and so $\mathbf{S}_\Theta^*=\mathbf{S}^*\oplus \mathbf{Z}^*|\mathbf{K}_\Theta$.

An important particular case is obtained when $\dim\DD_T=\dim\DD_{T^*}=1$, and the characteristic function is a scalar \emph{inner} function $\theta$. The model space is then $K_\theta=H^2\ominus \theta H^2$, and we will call $S_\theta\in\LL(K_\theta)$ a \emph{scalar model operator}.

Two operator valued analytic functions $\Theta, \Theta'$ defined in $\bbD$ are said to \emph{coincide}  if there are unitaries $\tau, \tau'$ such that $\Theta'=\tau\Theta\tau'$. Then two completely nonunitary contractions are unitarily equivalent if and only if their characteristic functions coincide.

We will use the relation between invariant subspaces and characteristic functions developed in the general case in~\cite[Chapter VII]{SF}; since we do not need the general theory, we single out in Lemma~\ref{le:factorization} below the precise consequences that we will need. 
In short, if $H'\subset H$ is an invariant subspace with respect to $T$, the decomposition of $T$ with respect to $H'\oplus H'{}^\perp$ being
\[
T= \begin{pmatrix}
T_1 & X\\ 0 & T_2
\end{pmatrix},
\]
there is an associated factorization of the characteristic function $\Theta$ such that 
\begin{equation}\label{eq:factorization}
\Theta=\Theta_2\Theta_1,
\end{equation}
where the characteristic function of $T_i$ is the pure part of $\Theta_i$. Such factorizations satisfy a supplementary condition of regularity (see~\cite[Theorem VII.1.1]{SF});  conversely, any factorization that satisfies this condition is obtained in this way from an invariant subspace.

%but this needs a few supplementary notions. First, a contractive analytic function $\Theta:\bbD\to \LL(E_1, E_2)$ can be decomposed as 
%\[
%\Theta(\lambda)= \begin{pmatrix}
%\Theta_0(\lambda) & 0\\ 0 & W
%\end{pmatrix}: E_1'\oplus E_1''\to E_2'\oplus E_2'',
%\]
%where $E_i=E_i'\oplus E_i''$ ($i=1,2$), $W$ is a unitary constant, and $\Theta_0$ is \emph{pure}, that is, it has no constant part; $\Theta_0$ is called the pure part of $\Theta$. 
%
%Secondly, we will be interested in factorizations of the type $\Theta=\Theta_2\Theta_1$, where $\Theta_1:\bbD\to \LL(E_1, E')$, $\Theta_2:\bbD\to \LL(E', E_2)$. For such factorizations there is certain notion of regularity, which can be found in Ch.~VII of NF. Here we do not need the general case; it suffices to know that if $\Theta_1, \Theta_2$ are both inner then the factorization is regular.  

\begin{lemma}\label{le:factorization} Suppose $T\in\LL(H)$, $H'\subset H$ is invariant to $T$, and denote $T'=T|H'$.

\begin{enumerate}
\item

If $T$ has inner characteristic function $ \left(\begin{smallmatrix}
\phi_1\\ \phi_2
\end{smallmatrix}\right)$ and $T'$ has scalar characteristic function $\theta$, then $\theta$ is a common inner divisor of $\phi_1$ and $\phi_2$. Conversely, if $\theta$ is a common inner divisor of $\phi_1$ and $\phi_2$, then there exists $H'\subset H$, invariant to $T$, such that the characteristic function of $T':=T|H'$ is~$\theta$.

\item
 If $T$ has scalar characteristic function $\Theta$ and $T'$ is an isometry, then $T'$ is a shift of multiplicity~1.

\end{enumerate}

\end{lemma}

\begin{proof}
We give just a sketch of the proof, based on the results in~\cite{SF}. 

(1) If $\Theta= \left(\begin{smallmatrix}
\phi_1\\ \phi_2
\end{smallmatrix}\right)$, it follows from~\eqref{eq:factorization} that $\Theta_1$ is a column of scalars, and thus it has to be actually the scalar fuction $\theta$ (there is no room for a constant unitary). Therefore  
\begin{equation}\label{eq:facto phi-i theta}
\begin{pmatrix}
\phi_1\\ \phi_2
\end{pmatrix}=
\begin{pmatrix}
\psi_1\\ \psi_2
\end{pmatrix}\theta,
\end{equation}
whence $\theta $ is a common inner divisor for $\phi_1$ and $\phi_2$. 

Conversely, if $\theta $ is a common inner divisor for $\phi_1$ and $\phi_2$, then~\eqref{eq:factorization} is true for some $\psi_i$; also, the factorization~\eqref{eq:factorization} is regular, since all functions are inner~\cite[Proposition VII.3.3]{SF}. Therefore $\theta$ is the characteristic function of a restriction of $T$ to an invariant subspace.

(2) This part follows immediately from the description of all factorizations of scalar characteristic functions given in~\cite[Proposition VII.3.5]{SF}. 
\end{proof}

\subsection{de Branges--Rovnyak spaces}

Suppose $b\in H^\infty$, $\|b\|_\infty\le 1$, and $b$ is nonextreme; $\Delta=(1-|b|^2)^{1/2}$, $a$ is the outer function that satisfies $|a|=\Delta$. $S$ is the unilateral shift on $H^2$, $Z$  the bilateral shift on $L^2$. We use the notation $\tilde f(z)=\overline{f(\bar z)}$.

Denote by $\Toe_b$ the Toeplitz operator with symbol $b$. The de Branges--Rovnyak space $\HH(b)$ is defined to be the range of $(I-\Toe_b\Toe_b^*)^{1/2}$, with the norm given by
\[
\|(g\|_{\HH(b)}=\inf\{\|f\|_2:(I-\Toe_b\Toe_b^*)^{1/2}f=g\}.
\]
In particular, if $\ker (I-\Toe_b\Toe_b^*)^{1/2}=\{0\}$, then
\begin{equation}\label{eq:norm in h(b)}
\|(I-\Toe_b\Toe_b^*)^{1/2}f\|_{\HH(b)}=\|f\|_2.
\end{equation}

In the sequel we will suppose that $b$ is nonextreme.  
Since $\Toe_b\Toe_b^*\le \Toe_b^*\Toe_b$, we have, for each $f\in H^2$, 
\[
\|(I-\Toe_b\Toe_b^*)^{1/2}f\|_2^2=\|f\|_2^2-\|\Toe_b^* f\|_2^2
\ge \|f\|_2^2-\|\Toe_b f\|_2^2.
\]
But, if $b$ is nonextreme, then $|b|<1$ a.e., whence, for $f\not\equiv0$, $\|\Toe_b f\|_2<\|f\|_2$. Therefore $\ker (I-\Toe_b\Toe_b^*)^{1/2}=\{0\}$ and~\eqref{eq:norm in h(b)} is satisfied.

It is proved in~\cite[II-7]{Sa} that $\HH(b)$ is invariant with respect to $S^*$, which acts as a contraction on $\HH(b)$. This contraction is denoted by $X_b$; it will be the main character in the sequel, but only in disguise.

Some spaces that will appear in the sequel are:
\begin{align*}
\KK_b&=(H^2\oplus \overline{\Delta H^2}) \ominus \{ bh\oplus \Delta h: h\in H^2\},\\
\tilde\KK_b&=(H^2\oplus L^2) \ominus \{ bh\oplus \Delta h: h\in H^2\},\\
\JJ_b&=\tilde\KK_b\ominus \KK_b=\{0\}\oplus (L^2\ominus \overline{\Delta H^2}),\\
Y_b&= P_{\KK_b} (S^*\oplus Z^*)|\KK_b,\\
\mathbf{Y}_b&= S^*\oplus Z^*|\tilde\KK_b.
\end{align*}

A basic reason why we introduce them is the next lemma.

\begin{lemma}\label{le:dBR-NF}
The orthogonal projection onto the first coordinate is a unitary operator from $\KK_b$ onto $\HH(b)$, that intertwines $Y_b$ with $X_b$.
\end{lemma}

\begin{proof}
The lemma is almost completely proved in~\cite[IV-7]{Sa}. It is shown therein that the operator
\[
B=\begin{pmatrix}
\Toe_b\\ -\Toe_a
\end{pmatrix}
\]
is an isometry from $H^2$ to $H^2\oplus H^2$, and that the projection $Q$ onto the first coordinate is a unitary from $\K_B=(H^2\oplus H^2)\ominus BH^2$ onto $\HH(b)$, which intertwines the restriction of $S^*\oplus S^*$ to this subspace with $X$. On the other hand, the map $W:\overline{\Delta H^2}\to H^2$ defined by $W(\Delta h)=-ah$ is easily seen to be an isometry, and it is actually unitary since $a$ is outer. It also commutes with $S$ and therefore, being unitary, with $S^*$.  Then $Q\circ (I_{H^2}\oplus W)$ yields the desired unitary operator.
\end{proof}

As a consequence, we will concentrate on  $Y_b$ rather than on $X_b$ in the rest of this paper.

The following result gathers some of the properties of the above spaces and operators. They constitute the basis for the ``model theory'' that will be investigated in the rest of the paper.

\begin{lemma}\label{le:Y_b isometries} Suppose $b\not\equiv0$. With the above notations, the following are true.
\begin{enumerate}
\item
We have  $\dim\DD_{Y_b}=2$, $\dim\DD_{Y_b^*}=1$, and $\dim\ker Y_b=1$. 
$Y_b$ is unitarily equivalent to $X_b=S^*|\HH(b)$.   Its characteristic function is 
\[
\Theta_{Y_b}= \begin{pmatrix}
\tilde a &\tilde b
\end{pmatrix}.
\]
Consequently, $Y_b\to 0$ strongly.
\item
$\mathbf{Y}^*_b$ is precisely the Nagy--Foias model corresponding to the characteristic function $ b$.

\item
$\mathbf{Y}_b$ is a nonisometric  dilation of $Y_b$; that is, it satisfies for all $n\in\bbN$ the relation $Y_b^n=P_{\KK_b}\mathbf{Y}_b^n|\KK_b$.

\item
$\mathbf{Y}_b|\JJ_b$ is an isometry. 
If $\XX\subset \tilde\KK_b$ is an invariant subspace for $\mathbf{Y}_b$, such that $\mathbf{Y}_b|\XX$ is an isometry, then $\XX\subset\JJ_b$. In particular, 
$\mathbf{Y}_b|\JJ_b$ is a maximal isometry contained in $\mathbf{Y}_b$, and
$Y_b$ has no isometric restriction.

\end{enumerate}

\end{lemma}

\begin{proof}

(1) The claimed properties of $Y_b$ are proved explicitely in~\cite[IV-7]{Sa} for $X_b$. The only exception is the dimension of $\ker Y_b$. Since $Y_b(\DD_{Y_b})\subset \DD_{Y_b^*}$, it has a nonzero kernel. If $f\oplus g\in \ker Y_b$, then $S^*f=0$, whence $f=c$ (constant). If we had two linearly independent vectors in $\ker Y_b$, some linear combination would have 0 as first coordinate, and thus we would have $0\oplus g_0\in \KK_b$ for some $g_0\in\overline{\Delta H^2}$, $g_0\not=0$. But the definition of $\KK_b$ implies $g_0\perp \Delta h$ for any $h\in H^2$, whence $g_0=0$, which is a contradiction.

(2) is an immediate consequence of the general form of the Sz.-Nagy--Foias model, while (3) follows easily from the fact that $\JJ_b$ is invariant with respect to $\mathbf{Y}_b$.

(4) It is immediate that $\mathbf{Y}_b|\JJ_b$ is an isometry, since it is unitarily equivalent to a restriction of $Z^*$. Then, if $\XX$ has the stated properties, take
$f\oplus g\in \XX$. We have 
\[
\|S^*{}^n f\|^2+\|Z^*{}^n g\|^2=\|\mathbf{Y}_b^n(f\oplus g)\|^2
=\|f\oplus g\|^2=\|f\|^2+\|g\|^2
\]
for any $n$. Since $Z^*$ is unitary and $S^*f\to 0$, this implies $f=0$, whence $\XX\subset \JJ_b$.
\end{proof}

It should be kept in mind that, according to Lemma~\ref{le:Y_b isometries}(2), the model operators in the de Branges--Rovnyak and Sz.Nagy--Foias approaches are mutual adjoints.

\section{A functional reformulation}

As noted in the introduction, we intend to find necessary and sufficient conditions for a c.n.u. contraction $T\in\LL(H)$ to be unitarily equivalent to $Y_b$ for some nonextreme function $b\in H^\infty$, $\|b\|\le 1$. Some necessary conditions follow already from Lemma~\ref{le:Y_b isometries}(1): we must have 

\smallskip

({\bf C1}) $\dim\DD_T=2, \dim\DD_{T^*}=\dim\ker T=1$, 

\smallskip

({\bf C2}) $T^n\to0$ strongly.
\smallskip

We will see later that (C1) and (C2) are not sufficient, but first we will use them in order to give an alternate formulation of the problem.

A general c.n.u. contraction $T $ with $\dim\DD_T=2, \dim\DD_{T^*}=1$ has as characteristic function an arbitrary pure contractive analytic function $\Theta_T:\bbD\to \LL(\bbC^2, \bbC)$
\begin{equation}\label{eq:Theta_T}
\Theta_T= \begin{pmatrix}
\phi_1& \phi_2
\end{pmatrix}.
\end{equation}
In this case the purity condition (which means that $\Theta_T$ has no constant unitary part) is equivalent to the fact that $\Theta_T$ does not coincide with the constant function
\[
\begin{pmatrix}
0 & \kappa
\end{pmatrix},
\]
where $\kappa\in\bbC$, $|\kappa|=1$.
Moreover, the condition $T^n\to 0 $ strongly is known to be equivalent to the identity $|\phi_1|^2+|\phi_2|^2=1$ (one says that $\Theta_T$ is \emph{$*$-inner}).

\begin{theorem}\label{pr:reformulation}
If $T$ is a c.n.u. contraction with characteristic function given by~\eqref{eq:Theta_T}, then the following are equivalent:
\begin{enumerate}
\item
$T$ is unitarily equivalent with $Y_b$ for some nonextreme b.
\item
$|\phi_1|^2+|\phi_2|^2=1$ a.e., and
there exist $\alpha_1, \alpha_2\in\bbC$ with $|\alpha_1|^2+|\alpha_2|^2=1$, such that  $\tilde a:=\alpha_1\phi_1+\alpha_2 \phi_2$ is an outer function.
\end{enumerate}
\end{theorem}

\begin{proof}
If (1) is true, then  $\Theta_T$ coincides with $\Theta_{Y_b}$. Using Lemma~\ref{le:Y_b isometries}(1) it follows that there exists a constant unitary $2\times 2$ matrix $\left(\begin{smallmatrix}\alpha_1&\alpha_3 \\[3pt]\alpha_2  &\alpha_4\end{smallmatrix}\right)$ such that
\begin{equation}\label{eq:condition phi1 phi2}
\begin{pmatrix}
\tilde a &\tilde b
\end{pmatrix} = \begin{pmatrix}
\phi_1& \phi_2
\end{pmatrix} \begin{pmatrix}\alpha_1&\alpha_3 \\\alpha_2  &\alpha_4\end{pmatrix}.
\end{equation}
Since $\tilde a$ is an outer function, (2) is proved.

Conversely, if (2) is true, then we may choose $\alpha_3, \alpha_4$ such that $\left(\begin{smallmatrix}\alpha_1&\alpha_3 \\[3pt]\alpha_2  &\alpha_4\end{smallmatrix}\right)$ is unitary, and~\eqref{eq:condition phi1 phi2} is satisfied with $\tilde b=\alpha_3 \phi_1+\alpha_4\phi_2$. Then $\tilde b$ is a function in the unit ball of $H^\infty$ that is nonextreme since $\int \log (1-|\tilde b|^2)=\int \log |\tilde a|^2>-\infty$. Since~\eqref{eq:condition phi1 phi2} and Lemma~\ref{le:Y_b isometries}(1)
 say  that $\Theta_T$ coincides with $\Theta_{Y_b}$, it follows that  $T$ is unitarily equivalent to $Y_{\tilde b}$.
\end{proof}

However, characterizing the pairs $(\phi_1, \phi_2)$ that satisfy~(2) seems an even more difficult problem.  Moreover, such a characterization would not use directly properties of the operator $T$, but rather of its characteristic function. That is why we seek other alternatives.

\section{Some necessary conditions}\label{sse:necessary}

As noted above, the conditions $\dim\DD_T=2$, $\dim \DD_{T^*}=\dim\ker T=1$, and $T^n\to 0$ strongly are necessary for the unitary equivalence of $T$ with some $Y_b$. They are not sufficient; a less obvious condition   is given by the next lemma.

\begin{lemma}\label{le:condition (2)}
If $T$ is unitarily equivalent to $Y_b$ for some nonextreme $b$, then:
\smallskip

{\bf (C3)} There is no subspace $Y$ of $H$ invariant with respect to $T^*$, such that $T^*|Y$ is unitarily equivalent to a scalar model operator.
\end{lemma}

\begin{proof}
Suppose $T$ has characteristic function given by~\eqref{eq:Theta_T}; then the characteristic function of $T^*$ is the inner function $\tilde\Theta_T=\left( \begin{smallmatrix}
\tilde\phi_1\\ \tilde \phi_2
\end{smallmatrix}\right)$. 
If $T$ is unitarily equivalent to $Y_b$ for some nonextreme $b$, it follows from Theorem~\ref{pr:reformulation} that  $\phi_1$ and $\phi_2$ must not have an inner common factor; the same is true also for $\tilde\phi_1, \tilde\phi_2$. The statement is then a consequence of Lemma~\ref{le:factorization}(1).
\end{proof}

It is easy now to give an example of an operator that satisfies (C1) and (C2) but not (C3): take $T=S^*\oplus S_\theta$ for some inner function $\theta$.

However, even all three conditions (C1--C3)
 are still not sufficient. To show this, it is enough, in view of Theorem~\ref{pr:reformulation} and Lemma~\ref{le:factorization}(1), to find   two functions $\phi_1, \phi_2\in H^\infty$ with $|\phi_1|^2+|\phi_2|^2=1$, such that $\phi_1$ and $\phi_2$ have no common inner factor and there is no linear combination of $\phi_1$ and $\phi_2$ which is outer. This is given in the next example.

\begin{example}\label{ex:(1)-(2) not sufficient}

Take $\phi_1(z)=\frac{1}{\sqrt2}z^2$, $\phi_2(z)=\frac{1}{\sqrt2}\frac{z-a}{1-\bar a z}$. Then obviously $|\phi_1|^2+|\phi_2|^2=1$. We will show that at least for $0<a<1/8$ there is no outer linear combination of $\phi_1$ and $\phi_2$, and thus, by Theorem~\ref{pr:reformulation}, $T$ is not unitarily equivalent with some $Y_b$ for $b$ nonextreme.

First, since $\phi_1, \phi_2$ themselves are not outer, it is enough to consider linear combinations of the type $\phi_1+\alpha\phi_2$ for some $\alpha\in\bbC$. If $|\alpha|<1$, then $|\alpha\phi_2(z)|<|\phi_1(z)|$ for $z\in\bbT$, and thus Rouch\'e's Theorem says that $\phi_1+\alpha\phi_2$ has the same number of zeros in $\bbD$ as $\phi_1$, so it cannot be outer. A similar argument settles the case $|\alpha|>1$.

Let us now consider $|\alpha|=1$. The equation $\phi_1(z)+\alpha\phi_2(z)=0$ can be written
\begin{equation}\label{eq:rouche}
z^2+\alpha z - a \left( \alpha +\frac{\bar a}{a} z^3\right)=0.
\end{equation}
If $|z|=1/2$, then $|z^2+\alpha z|=|z(z+\alpha)|>\frac{1}{2}\cdot \frac{1}{2}=\frac{1}{4}$. On the other hand, $|\alpha +\frac{\bar a}{a} z^3|\le 1+\frac{1}{8}<2$, and thus $|a \left( \alpha +\frac{\bar a}{a} z^3\right)|<\frac{1}{4}$ if $0<a<1/8$. We may again apply Rouch\'e's Theorem to conclude that~\eqref{eq:rouche} has a solution in the disc $\{|z|<1/4\}$, and thus neither is $\phi_1+\alpha\phi_2$ outer if $|\alpha|=1$.
\end{example}

We have then to find some other necessary condition, besides (C1)--(C3). This requires a certain construction that will be done in the next section.

\section{Construction of certain dilations}\label{se:the construction}

We start with an elementary lemma, whose proof we omit.

\begin{lemma}\label{le:2 x 2 lemma}
Suppose $0< \alpha<1$, $\xi=(\xi_1,\xi_2)$ with $|\xi_1|^2 + |\xi_2|^2=1$, and denote
\[
A:= \begin{pmatrix}
\alpha & 0\\
a\bar\xi_1 & a\bar \xi_2
\end{pmatrix}
\]
Then  
\begin{equation}\label{eq:formula for a}
a=a_\xi:= \left(
\frac{1-\alpha^2}{1-\alpha^2|\xi_2|^2}
\right)^{1/2},
\end{equation} 
is the only value of $a$ for which $A$
is a contraction with $\dim\DD_A=\dim\DD_{A^*}=1$. If we denote then by $e_\xi$ a unit vector in  $\ker (I-A^*A)$ (therefore $\|A e_\xi\|=\|e_\xi\|$), then $e_\xi$ is determined up to a unimodular constant; moreover, for 
 any $\eta\in\bbC^2$ with $\|\eta\|=1$, there exists $\xi=(\xi_1, \xi_2)$, $|\xi_1|^2 + |\xi_2|^2=1$ such that $e_\xi=\eta$.
\end{lemma}

Note that for  $a<a_\xi$ the defects have dimension~2, while for $a>a_\xi$ $A$ is no more a contraction. Also, we may take $e_\xi=\xi$ if and only if $\xi$ is one of the standard basis vectors.

To go  beyond the  conditions in Section~\ref{sse:necessary}, we consider a construction that stems from the fact that $\mathbf{Y}_b$ is a nonisometric  dilation of $Y_b$. We have then to discuss a certain general construction of nonisometric  dilations.
Suppose then that $T$ is a contraction acting on the Hilbert space $H$ with $\dim\DD_T=2$, $\dim\DD_{T^*}=\dim\ker T=1$.

We are interested in  dilations $\tilde T$ of $T$ with the property that $\dim\DD_{\tilde T}= \dim\DD_{\tilde T^*}=1$. These may be described in the following manner. For clarity of notation, we will denote by $T_d$ and $T_u$ the restrictions $T_d=T:\DD_T \to \DD_{T^*}$, $T_u=T:\DD_T^\perp \to \DD_{T^*}^\perp$; note that $T_u$ is unitary and $T_d$ is a strict contraction.

Take a vector  $\xi\in \DD_T$, with   $\|\xi\|=1$, and consider, for $0<a\le 1$, the operator
\[
A_\xi := \begin{pmatrix}
T_d\\ a_\xi\otimes \xi 
\end{pmatrix}:\DD_T\to \DD_{T^*}\oplus \bbC.
\]
The operator $T_d$ is a strict contraction with kernel of dimension~1.
If we choose in $\DD_T$ a basis formed by the eigenvectors of $T_d^*T_d$, then the matrix of $T_d$ is $(\alpha\quad 0)$ for some $0<\alpha<1$, and thus
 $A_\xi$ is precisely the $A$ in Lemma~\ref{le:2 x 2 lemma}. 
Consequently, $\dim \DD_{ A_\xi}=\dim \DD_{ A_\xi^*}=1$.
Remember that $e_\xi$ is a normalized vector in $\ker(I-A_\xi^*A_\xi)\cap \DD_T$; this notation will be used consistently in the sequel of the paper.

Consider then the space
\[
K_\xi=H\oplus\bbC\oplus\bbC\oplus \cdots
\]
on which acts the operator
\begin{equation}\label{eq:T xi}  T_\xi:=
\begin{pmatrix}
T & 0 & 0 &0&\dots\\
a_\xi\otimes \xi &0 &0&0&\dots\\
0& 1 & 0 &0&\dots\\
0 & 0& 1 & 0&\dots\\
\vdots&\vdots&\vdots&\ddots
\end{pmatrix}.
\end{equation}

\begin{lemma}\label{le:one dilations are T_xi}
\begin{enumerate}
\item

 With the above notations, $T_\xi$ is a minimal contractive dilation of $T$ satisfying $\dim \DD_{  T_\xi}=\dim \DD_{T^*_\xi}=1$.  

\item
 Suppose $\widehat T\in \LL(\widehat H)$ is a dilation of $T$, such that $\dim \DD_{  \widehat T}=1$ and $\widehat T|\widehat H\ominus H$ is a pure isometry of multiplicity~1. Then $\widehat T$ is unitarily equivalent to some~$T_\xi$ as above. 
 
\end{enumerate}
\end{lemma}

\begin{proof}
(1) Let us denote $\II_\xi=K_\xi\ominus H$. With respect to the two decompositions
\[
K_\xi=\DD_T^\perp \oplus \DD_T\oplus \bbC\oplus\bbC\oplus \cdots =
\DD_{T^*}^\perp \oplus \DD_{T^*}\oplus \bbC\oplus\bbC\oplus \cdots
\]
$T_\xi$ has the matrix
\begin{equation}\label{eq:T xi decomposed further}  T_\xi:=
\begin{pmatrix}
T_u & 0 & 0 &0&\dots\\
0&T_d & 0 & 0 &\dots\\
0& a_\xi\otimes \xi &0 &0&\dots\\
0& 0& 1 & 0 &\dots\\
0& 0 & 0& 1 & \dots\\
\vdots&\vdots&\vdots&\ddots
\end{pmatrix},
\end{equation}
which means that, if in the range space we consider together the second and the third space ($\DD_{T^*}\oplus\bbC$), the matrix of $T_\xi$ is diagonalized, and we have
\[
T_\xi=T_u\oplus A_\xi\oplus 1\oplus 1\oplus\cdots.
\] 
Since all operators except the second are unitary, $T_\xi$ is a contraction and the dimensions of its defects are the same as those of $A_\xi$, that is~1.
Moreover, $T_\xi$ is a minimal dilation of $T$.

(2) Suppose $\widehat T$ is a dilation of $T$ with $\dim\DD_{\widehat T}=1$, acting on $\widehat H\supset H$. Since $\widehat T|\widehat H\ominus H$ is a shift of multiplicity~1, $\widehat T$ must have the form
\[
 \widehat T:=
\begin{pmatrix}
T & 0 & 0 &\dots\\
X &0 &0&\dots\\
0& 1 & 0 &\dots\\
0 & 0& 1 & \dots\\
\vdots&\vdots&\vdots&\ddots
\end{pmatrix},
\]
with a nonnull $X:H\to \bbC$, $X=a\otimes\xi$ for some $\xi\in H$ with $\|\xi\|=1$ and some $a$.  We have $\DD_{\widehat T}\subset H$, and $D_{\widehat T}|H=I-T^*T-X^*X$. Since $I-T^*T$ has rank 2, while $I-T^*T-X^*X$ has rank~1, it follows from Lemma~\ref{le:2 x 2 lemma} that  $X=a_\xi\otimes\xi$, with $\xi\in\ker D_{\widehat T}\cap \DD_T$.
\end{proof}

\begin{lemma}\label{le:properties of Txi}

If $\xi$ is an eigenvector of $\DD_T$, and $T_\xi$ is completely nonunitary, then:

\begin{enumerate}
\item
The characteristic function $b_\xi$ of $T_\xi$ is nonextreme. 

\item
If $\II_\xi\subset \YY\subset K_\xi$, $\YY$ is invariant with respect to $T_\xi$, and $T_\xi|\YY$ is an isometry, then $T_\xi|\YY$ is a shift of multiplicity~1.

\end{enumerate}
\end{lemma}

\begin{proof}
By Lemma~\ref{le:one dilations are T_xi}, we have $\dim \DD_{  T_\xi}=\dim \DD_{  T^*_\xi}=1$, so $T_\xi$ has a scalar characteristic function $b_\xi$. This has to be nonextreme since $T_\xi$ has an isometric restriction (namely, $T|\II_\xi$).

For the second statement, apply Lemma~\ref{le:factorization}(2) to the contraction $T_\xi$ and its invariant subspace $\YY$.
\end{proof}

At this point we may give  another reformulation of the main question.

\begin{theorem}\label{pr:reformulation 2}
If $T$ is a c.n.u. with characteristic function given by~\eqref{eq:Theta_T}, then the following are equivalent:
\begin{enumerate}
\item
$T$ is unitarily equivalent with $Y_b$ for some nonextreme $b$.
\item
There exists $\xi\in\DD_T$, $\|\xi\|=1$, such that the contraction $T_\xi$ defined by~\eqref{eq:T xi} is completely nonunitary and $T_\xi|\II_\xi$ is a maximal isometry.
\end{enumerate}
\end{theorem}

\begin{proof}
If $T$ is unitarily equivalent with $Y_b$ for some nonextreme $b$, then, by~Lemma~\ref{le:Y_b isometries}, $\mathbf{Y}_b$ is a completely nonunitary dilation of $Y_b$ with the required properties in the assumptions of~\ref{le:one dilations are T_xi}(2), whence it has to be unitarily equivalent to some $T_\xi$. By~Lemma~\ref{le:Y_b isometries} we know that $T_\xi|\II_\xi$ is a maximal isometry.

Conversely, if (2) is true, the given completely nonunitary contraction $T_\xi$ has a nonextreme characteristic function $b_\xi$ by Lemma~\ref{le:properties of Txi}(i). There exists therefore a unitary $W:K_\xi\to \widehat\KK_{b_\xi}$, such that $\mathbf{Y}_{b_\xi}W=WT_\xi$. By Lemma~\ref{le:Y_b isometries}(4), $\JJ_b$ is the space on which acts the unique maximal isometry contained in $\mathbf{Y}_{b_\xi}$, and therefore it has to be equal to $W\II_\xi$. Passing to orthogonals, $W$ maps $H$ onto $\KK_{b_\xi}$, and commutes with the respective compressions there. This says precisely that $T$ is unitarily equivalent to $Y_{b_\xi}$.   
\end{proof}

We have then to investigate the two properties in point (2) of the above proposition.

\section{$T_\xi$ completely nonunitary}

We prove in this section that conditions (C1)--(C3) imply that $T_\xi$ is  completely nonunitary. 

\begin{proposition}\label{pr:T xi completely nonunitary}
Suppose $T$ is a c.n.u. contraction on $H$ that satisfies conditions (C1)--(C3).
Then $T_\xi$ is  completely nonunitary for all $\xi\in\DD_T$, $\|\xi\|=1$.

\end{proposition}

\begin{proof}
Denote by $V$ the minimal isometric dilation of $T_\xi$, acting on the space $K\supset H$. Since $T_\xi$ is a minimal dilation of $T$, it follows easily that $V$ is also a minimal isometric dilation of $T$. 

We will use the Sz.-Nagy--Foias model of the contraction $T$, which is the space
\[
\mathbf{H}=(H^2\oplus \Delta L^2(\bbC^2)) \ominus \{\Theta_T h\oplus \Delta h: h\in H^2(\bbC^2)\}
\]
and the operator unitarily equivalent to $T$ is $\mathbf{T}=P_{\mathbf{H}}(S\oplus Z)|\mathbf{H}$. The minimal unitary dilation $\mathbf V$ is just $S\oplus Z$ acting on $\mathbf{K}=H^2\oplus \Delta L^2(\bbC^2)$, and its unitary part acts on the space $\{0\}\oplus \Delta L^2(\bbC^2)$. 
Let us denote by $\Omega$ the unitary that implements the equivalence; that is, $\Omega:K\to \mathbf{K}$, $\Omega(H)=\mathbf{H}$, $\Omega V=\mathbf{V}\Omega$.

If $T^n\to 0$ strongly, then the characteristic function of $T$ is given by~\eqref{eq:Theta_T}, with $|\phi_1|^2+|\phi_2|^2=1$ a.e. Then 
\[
\Theta_T^*\Theta_T= \begin{pmatrix}
|\phi_1|^2& \bar\phi_1 \phi_2\\
\bar\phi_2\phi_1 & |\phi_2|^2
\end{pmatrix}= \begin{pmatrix}
\bar\phi_1\\ \bar\phi_2
\end{pmatrix} \begin{pmatrix}
\phi_1&\phi_2
\end{pmatrix}
\]
is almost everywhere on $\bbT$ a one-dimensional projection in $\bbC^2$. Therefore $\Delta(e^{it})$ is also a one-dimensional projection a.e.
If we write  $J(e^{it})=\begin{pmatrix}
\bar\phi_1(e^{it})\\ \bar\phi_2(e^{it})
\end{pmatrix}:\bbC\to\bbC^2$, then the map $f\mapsto J(f)$ is a unitary operator from $L^2$ to $JL^2=\Delta L^2(\bbC^2)$. Moreover, $J$ intertwines multiplication with $e^{it}$ in the corresponding $L^2$ spaces.

Consider now the operator $\mathbf{T}_\xi$ corresponding in the Sz.-Nagy--Foias model to  $T_\xi$, that is, $\mathbf{T}_\xi =\Omega T_\xi\Omega^*$.
Its unitary part is a reducing subspace of the unitary part of $\mathbf{V}$, and thus has to be a reducing subspace of $\{0\}\oplus \Delta L^2(\bbC^2)$ with respect to $S\oplus Z$, which means a reducing subspace of $JL^2$ with respect to multiplication by $e^{it}$. Therefore it is $J(L^2(E))$ for some measurable subset $E\subset\bbT$, or, equivalently, $\Delta L^2(E)$.

Consider now the vector $e_\xi$ introduced in the previous section. Since $\|T_\xi e_\xi\|=\|e_\xi\|$, we must also have $T_\xi e_\xi=Ve_\xi$, and therefore 
\begin{equation*}%\label{eq:where is T xi e_xi}
\mathbf{T}_\xi \Omega e_\xi=\mathbf{V}\Omega e_\xi =(S\oplus Z)e_\xi\in (S\oplus Z)\mathbf{H}\subset \mathbf{H}\oplus \{\Theta_T c_1\oplus \Delta c_2: c_1, c_2\in \bbC\}.
\end{equation*}

By~\eqref{eq:T xi}, $T_\xi e_\xi$ belongs to $H\oplus \bbC$ (it has no components on the subsequent copies of $\bbC$ in the formula of $K_\xi$), and the second component is $a_\xi\not=0$.
So the projection of $T_\xi e_\xi$ onto $\II_\xi$ is a nonzero vector on the first component of $\II_\xi$, which is a wandering vector for $T_\xi|\II_\xi$. Applying $\Omega$ to this projection, we obtain that a wandering vector for $\mathbf{T}_\xi|\Omega(\II_\xi)$ is of the form $\Theta_T c_1\oplus \Delta c_2$. After a change of basis in $\DD_T$, we may assume that $c_2=0$.

It follows then that $\mathbf{T}_\xi$ is the compression of $S\oplus Z$ to the space
\begin{equation}\label{eq:K_xi in the model}
\begin{split}
\mathbf{K}_\xi&= \mathbf{H}\oplus \left\{\Theta_T \begin{pmatrix}
h\\0
\end{pmatrix}\oplus \Delta \begin{pmatrix}
h\\ 0
\end{pmatrix}:  h\in H^2\right\}\\&= \mathbf{K}\ominus \left\{\Theta_T \begin{pmatrix}
0\\h
\end{pmatrix}\oplus \Delta \begin{pmatrix}
0\\ h
\end{pmatrix}:  h\in H^2\right\}.
\end{split}
\end{equation}

Now, if $\{0\}\oplus\Delta L^2(E)\subset \mathbf{K}_\xi$, it has to be orthogonal to $\Delta \left(\begin{smallmatrix}
0\\ H^2
\end{smallmatrix}\right)$, whence $\Delta(e^{it})$ must be a.e. on $E$ the projection on the first coordinate. That means that $\phi_1=0$ a.e. on $E$, whence $\phi_1\equiv0$, $\phi_2$ inner. This is excluded by  the last part of the hypothesis.
\end{proof}

\section{The final  result}

We need only one more ingredient to obtain the final result.

\begin{lemma}\label{le:maximal isometry}
Suppose $T$ is a c.n.u. contraction on $H$ that satisfies conditions (C1)--(C3).   The following are equivalent:
\begin{enumerate}
\item
$T_\xi|\II_\xi$ is a maximal isometry.
\item
For any $H'\subset H$ such that $TH'\subset H'$ and $T':=T|H'$ is a scalar model operator, we have $e_\xi\notin H'$.
\item
For any $H'\subset H$ such that $TH'\subset H'$ and $T':=T|H'$ is a scalar model operator, we have $e_\xi\notin \DD_{T'}$.
\end{enumerate}
\end{lemma}

\begin{proof}
(1)$\implies$(2). Suppose (1) is true, and let $H'\subset H$ such that  $TH'\subset H'$, $e_\xi\in H'$, and $T':=T|H'$ is a scalar model operator. Then $\DD_{T'}$ having dimension~1, is spanned by $e_\xi$.
It may then be checked that the space $\YY=\II_\xi\oplus H'$ is invariant with respect to $T_\xi$, and $T_\xi|\YY$ is an isometry that strictly extends $T_\xi|\II_\xi$. Therefore  $e_\xi\notin H'$.

(2)$\implies(3)$ is immediate. Let us assume that (3) is true, and suppose $\II_\xi\subset \YY\subset \KK_\xi$, $T_\xi\YY\subset \YY$, and $T_\xi|\YY$ is an isometry. If $\YY'=\YY\cap H\not=\{0\}$ and $T'=P_{\YY'}T'_\xi|\YY'$, then $T'_\xi$ is an isometric dilation of $T'$, which is a shift of multiplicity 1 by Lemma~\ref{le:properties of Txi}(2). Thus $T'$ is the compression of a shift of multiplicity one to a coinvariant subspace, which is precisely unitarily equivalent to a scalar model operator. 

Since $\DD_{T'}=\{x\in \YY':\|T'x\|<\|x\|\}$, we have $\DD_{T'}\subset \DD_T$. 
Suppose then $x\in \DD_{T'}$, $x=x_1+x_2$, with $x_1\in \ker T$, $x_2$ multiple of $e_\xi$. We have then
\[
\|x_1\|^2+\|x_2\|^2=\|x\|^2=\|T_\xi x\|^2= \|Tx_1\|^2+\|Tx_2\|^2\le \|x_2\|^2,
\]
 whence $x_1=0$. Therefore $x$ is a multiple of $e_\xi$, which contradicts assumption~(3). It follows that $\YY=\II_\xi$, ending the proof of the lemma.
\end{proof}

In the light of Lemma~\ref{le:maximal isometry}, we may now state the last necessary condition:
\smallskip

({\bf C4}) There exists $\eta\in\DD_T$ such that, if $\YY'\subset H$, $T\YY'\subset \YY'$, and $T':=T|\YY'$ is unitarily equivalent to a scalar model operator,  then $\eta\notin\YY'$. 
\smallskip

The desired characterization is then given by the next theorem.

\begin{theorem}\label{th:main}
Suppose $T$ is a c.n.u. contraction on $H$.  The following are equivalent:
\begin{enumerate}
\item
$T$ is unitarily equivalent to $X_b$ for some nonextreme function $b\not\equiv0$.

\item
$T$ satisfies conditions (C1)--(C4).

\end{enumerate}

\end{theorem}

\begin{proof}

If $T$ is unitarily equivalent to $X_b$ for some nonextreme function $b\not\equiv0$, then (C1)--(C3) have already been proved. To prove (C4), 
note that, since $\mathbf{Y}_b$ is a dilation of $Y_b$ with $\dim \DD_{Y_b}=1$ and $\mathbf{Y}_b|\tilde{K}_b\ominus K_b$ is a maximal isometry, it follows from Lemma~\ref{le:one dilations are T_xi}(2) that  $\mathbf{Y}_b$ is unitarily equivalent to $T_\xi$, with $\xi\in\DD_{Y_b}$. Then (C4) follows from Lemma~\ref{le:maximal isometry}.

For the reverse implication, choose a vector $\xi$ such that $\eta=e_\xi$; its existence is ensured by Lemma~\ref{le:2 x 2 lemma}.
The dilation $T_\xi$  is a completely nonunitary contraction by Proposition~\ref{pr:T xi completely nonunitary}. Lemma ~\ref{le:maximal isometry} ensures that $T_\xi|\II_\xi$ is a maximal isometry, and then Theorem~\ref{pr:reformulation 2} implies that (1) is true. 
\end{proof}

Condition (C3) can be reformulated as 

\smallskip

({\bf C3${}'$}) There exists no subspace $\YY\subset H$ such that $T^*\YY\subset \YY$ and, if $T_\YY:=T^*|\YY$, then $\dim\DD_{T_\YY} = \dim\DD_{T^*_\YY}=1$. 

\smallskip

Indeed, we have $T_\YY^*{}^n=P_\YY T^n\to 0$ strongly by (a).
Similarly, condition (C4) can be reformulated as 
\smallskip

({\bf C4${}'$}) There exists $\eta\in\DD_T$ such that, whenever $\YY'\subset H$, $T\YY'\subset \YY'$, and, if $T':=T|\YY'$, $\dim\DD_{T'}=\dim\DD_{T'{}^*}=1$,  we have $\eta\notin\YY'$.

\section{Freedom in the choice of $b$}

A natural question when considering model theory is whether a given operator determines its model (up to some simple transformation). 
Let us then suppose that a contraction $T\in\LL(H)$ is unitarily equivalent to $X_{b_1}$ as well as to $X_{b_2}$ for some $b_1, b_2$ in the unit ball of $H^\infty$. 
Since $X_{b_1}^*$ and $X_{b_2}^*$ have to be unitarily equivalent, their characteristic functions must coincide.
Also, by looking at the dimensions of the defect spaces of $T$, it follows immediately that  $b_1, b_2$ are simultaneously extreme or nonextreme, so we have to discuss two cases.

If $b$ is extreme, then the characteristic function of $X_b^*$ is precisely $b$. So the answer is simple: if $T$ is unitarily equivalent to $X_{b_1}$ as well as to $X_{b_2}$, then $b_1=\kappa b_2$ for some unimodular constant $\kappa$.

If $b_1, b_2$ are nonextreme, the characteristic functions of $X_{b_1}^*$ and $X_{b_2}^*$ are given by Lemma~\ref{le:Y_b isometries}, and if they coincide we must have 
\begin{equation}\label{eq:coincidence of two a b}
\begin{pmatrix}
a_2\\ b_2
\end{pmatrix} = 
\begin{pmatrix}
\alpha & \beta\\ \gamma & \delta
\end{pmatrix} \begin{pmatrix}
a_1\\ b_1
\end{pmatrix}
\end{equation}
for some unitary constant matrix $\left(\begin{smallmatrix}
\alpha & \beta\\ \gamma & \delta
\end{smallmatrix}\right)$.
This is possible for rather different functions $b_1, b_2$, as shown by the following example. Take $b_1=z/\sqrt{2}$ (so $a_1=1/\sqrt{2}$), and $\alpha=\beta=\gamma=-\delta=1/\sqrt{2}$; it follows  that $b_2=\frac{1-z}{2}$. We have then $X_{b_1}$ unitarily equivalent to $X_{b_2}$, but $b_1$ is a constant multiple of an inner function, while $b_2$ is outer. There seems to be no simple criterion that could decide when $X_{b_1}$ unitarily equivalent to $X_{b_2}$ without involving the associated outer functions $a_1$ and $a_2$.

A natural question is then whether there exist  cases when, as in the extreme case, $b$ is uniquely determined up to a unimodular constant. 
If $X_{b_1}$ and $X_{b_2}$ are unitarily equivalent, then~\eqref{eq:coincidence of two a b} implies, in particular, that $a_2=\alpha a_1+\beta b_1$ is outer.
Conversely, suppose $b_1$ is given, $a_1$ is the associated outer function, and a certain combination $a_2=\alpha a_1+\beta b_1$ is outer. We may suppose $|\alpha|^2+|\beta|^2=1$; if we take $\gamma=\bar\beta$, $\delta=-\bar\alpha$, then $\left(\begin{smallmatrix}
\alpha & \beta\\ \gamma & \delta
\end{smallmatrix}\right)$ is unitary and $b_2$ defined by~\eqref{eq:coincidence of two a b} has the property that $X_{b_2}$ is unitarily equivalent to $X_{b_1}$.

We may then reformulate the last problem as follows:

\smallskip
{\bf Question:} Does there exist a nonextreme function $b$ such that, if $a$ is the associated outer function, then $\alpha a+\beta b$ outer implies $\beta=0$?


\begin{thebibliography}{99}

\bibitem{A}
 D. Alpay, O. Timoshenko, P. Vegulla, D. Volok,  Generalized Schur functions and related de Branges-Rovnyak spaces in the Banach space setting \emph{Integral Equations Operator Theory} {\bf 65} (2009), 449--472.
 
\bibitem{BBH0} J.A. Ball, V. Bolotnikov, S. ter Horst,  Interpolation
in  de Branges-Rovnyak spaces, \emph{Proc. Amer. Math. Soc.} {\bf 139} (2011), 
609--618. 


\bibitem{BBH} J.A. Ball, V. Bolotnikov, S. ter Horst,  Abstract interpolation in vector-valued de Branges-Rovnyak spaces, \emph{Integral Equations Operator Theory} {\bf 70} (2011),  227--263. 



\bibitem{BK} J.A. Ball, Th.L. Kriete, 
Operator-valued Nevanlinna-Pick kernels and the functional models for contraction operators,
\emph{Integral Equations Operator Theory} {\bf 10} (1987), 17--61.

\bibitem{BFM} A. Baranov, E. Fricain, J. Mashreghi, Weighted norm inequalities for de Branges-Rovnyak spaces and their applications, \emph{Amer. J. Math.} {\bf 132} (2010), 125--155.


\bibitem{BR1}
L. de Branges, J. Rovnyak, Canonical models in quantum scattering theory, in \emph{Perturbation Theory and its Applications in Quantum Mechanics}, ed. by C.H. Wilcox, Wiley, 1966, 295--392.

\bibitem{BR}L. de Branges, J. Rovnyak,
\emph{
Square summable power series}, Holt, Rinehart and Winston,  1966.


\bibitem{CGR} N. Chevrot, D. Guillot, Th. Ransford, De Branges-Rovnyak spaces
and Dirichlet spaces, \emph{J. Funct. Anal.} {\bf 259} (2010), 2366--2383.

\bibitem{CR} C. Costara, Th. Ransford, Which de Branges--Rovnyak spaces are
Dirichlet spaces (and vice versa)?, \emph{J. Funct. Anal.} {\bf 265} (2013),
3204--3218.

\bibitem{F}
E. Fricain, Bases of reproducing kernels in de Branges spaces, \emph{J. Funct. Anal.} {\bf 226}  (2005), 373--405.

\bibitem{FM1} E. Fricain, J. Mashreghi, Boundary behavior of functions in the de Branges-Rovnyak spaces \emph{Complex Anal. Oper. Theory} {\bf 2} (2008),  87--97.

\bibitem{FM2} E. Fricain, J. Mashreghi, Integral representation of the $n$-th derivative in de Branges-Rovnyak spaces and the norm convergence of its reproducing kernel, \emph{Ann. Inst. Fourier (Grenoble)} {\bf 58} (2008),  2113--2135.


\bibitem{FMb} E. Fricain, J. Mashreghi, \emph{Theory of $\HH(b)$ Spaces, vol. I
and II}, New Monographs in Mathematics, Cambridge University Press, to appear.

\bibitem{J} M.T. Jury, Reproducing kernels, de Branges-Rovnyak spaces, and norms of weighted composition operators, \emph{Proc. Amer. Math. Soc.} {\bf 135} (2007), 3669--3675. 

\bibitem{LS1} B.A. Lotto, D. Sarason, Multiplicative structure of de Branges's
spaces, \emph{Rev. Mat. Iberoamericana} {\bf 7} (1991), 183--220.

\bibitem{LS2} B.A. Lotto, D. Sarason, Multipliers of de Branges-Rovnyak spaces,
\emph{Indiana Univ. Math. J.} {\bf 42} (1993), 907--920.

\bibitem{LS3} B.A. Lotto, D. Sarason, Multipliers of de Branges-Rovnyak spaces,
II, in
\emph{Harmonic Analysis and Hypergroups}, Delhi, 1995, 51--58.

\bibitem{NV} 
 N.K. Nikolski, V.I. Vasyunin,  Notes on two function models, in \emph{The Bieberbach conjecture (West Lafayette, Ind., 1985)}, , Math. Surveys Monogr., 21, Amer. Math. Soc., Providence, RI, 1986, 113--141.

\bibitem{Sa}
D. Sarason, \emph{Sub-Hardy Hilbert Spaces in the Unit Disk}, John Wiley, 1994.


\bibitem{Sa1}
D. Sarason, Local Dirichlet spaces as de Branges-Rovnyak spaces, \emph{Proc.
Amer. Math. Soc.} {\bf 125} (1997), 2133--2139.

\bibitem{SF}
B. Sz.--Nagy, C. Foias, H. Bercovici, L. K\'erchy, \emph{Harmonic Analysis of Operators on Hilbert Space}, second edition, Springer, 2010.


\end{thebibliography}
\end{document}